\documentclass{amsart}%
\usepackage{amsmath}
\usepackage{amsfonts}
\usepackage{amssymb}
\usepackage{graphicx}%
\usepackage{wrapfig}
\usepackage{epsfig}
\newtheorem{lemma}{Lemma}[section]
\newtheorem{theorem}[lemma]{Theorem}
\newtheorem{remark}[lemma]{Remark}

\newtheorem{coro}[lemma]{Corollary}

\newtheorem{example}[lemma]{Example}

\email[T.~Caraballo]{caraball@us.es} \email[D.
Cheban]{cheban@usm.md} 

\title[Structure of the Global Attractor for Difference Equations]{On the Structure of
the Global Attractor for Non-autonomous Difference Equations with
Weak Convergence.}
\author{Tom\'{a}s Caraballo}
\address[T. Caraballo]{Departamento de Ecuaciones Diferenciales y An\'{a}lisis Num\'{e}rico\\
Universidad de Sevilla\\
Apdo. Correos 1160 \\
41080-Sevilla (Spain)}
\author{David Cheban}
\address[D. Cheban]{State University of Moldova\\
Department of Mathematics and Informatics\\
A. Mateevich Street 60\\
MD--2009 Chi\c{s}in\u{a}u, Moldova}
\date{\today}
\begin{document}

%%%%%%%%%%%%%%%%%%%%%%%%%%%%%%
%%%%%%%%%%%%%%%%%%%%%%%%%%%%%%
%%%%%%%%%%
\begin{abstract}
{The aim of this paper is to describe the structure of global
attractors for non-autonomous difference systems of equations with
recurrent (in particular, almost periodic) coefficients. We
consider a special class of this type of systems (the so--called
weak convergent systems). We study this problem in the framework
of general non-autonomous dynamical systems (cocycles). We apply
the general results obtained in our early papers to study the
almost periodic (almost automorphic, recurrent)
and asymptotically almost periodic (asymptotically almost
automorphic, asymptotically recurrent) solutions of difference equations.}

\end{abstract}

\maketitle

\begin{center}
{\it Dedicated to Francisco Balibrea on occasion of his 60$^{\it th}$ birthday}
\end{center}

\section{Introduction}

Denote by  $C(\mathbb R\times \mathbb
R^n,\mathbb R^n)$ the space of all continuous functions $f:\mathbb
R\times \mathbb R^n\to \mathbb R^n$ equipped with the
compact-open topology, and by $|\cdot|$ the norm in  $\mathbb R^n$.

Consider a system of differential equations
\begin{equation}\label{eqI1*}
x'=f(t,x),
\end{equation}
where $f\in C(\mathbb R\times \mathbb R^n,\mathbb R^n)$. Assume
that the right-hand side of (\ref{eqI1*}) satisfies hypotheses
ensuring the existence, uniqueness and extendability of solutions
of (\ref{eqI1*}),  i.e., for all $(t_0,x_0)\in \mathbb R\times
\mathbb R^n$ there exists a unique solution $x(t;t_0,x_0)$ of
(\ref{eqI1*}) with initial data $t_0,x_0$, and defined for all
$t\geq t_0$.

Recall (see, for example, \cite{Dem_67,Yosh_1975}) that equation
(\ref{eqI1*}) is said to be uniformly dissipative (or uniformly
ultimately bounded) if there exists a number $r_0>0$ so that, for
every $r>0$, there is  $L(r)>0$ such that if $|x_0|\le r$, then
$|x(t;t_0,x_0)|\le r_0$ if $t\ge t_0+L(r)$.

In this framework, we aim to provide some results on an interesting classical question which is due to Seifert. It is described in the following way:

\textbf{\emph{Problem}} (G. Seifert \cite{FF_1971}): Suppose that (\ref{eqI1*}) is uniformly dissipative and the function $f$ is
almost periodic (with respect to the time variable $t$). Does equation
(\ref{eqI1*}) possess an almost periodic solution?

Fink and Fredericson \cite{FF_1971} and Zhikov \cite{Zhi_1972}
established that, in general,  even when (\ref{eqI1*}) is scalar
($n=1$), the answer to Seifert's question is negative.

However, in view of this negative general answer, there are still several aspects which can be analyzed and provide useful information on this problem. On the one hand, it would be very interesting to investigate the existence of certain classes of dissipative differential equations for which the response to Seifert's question is affirmative. And, on the other hand, one could be interested in finding out some additional assumptions (\lq\lq optimal\rq\rq \ if possible) which guarantee the existence of at least one almost periodic solution (see \cite{CC_I} for a short description of already published results concerning these questions).

In our earlier articles \cite{CC_2010,CC_I}, we have proved some partial results on these problems. However, our main aim in the present paper is to analyze the same kind of questions but for dissipative almost periodic difference equations. To be more precise, we will investigate  Seifert's problem for the following difference
equation
\begin{equation}\label{eqi1}
u(t+1)=f(t,u(t)), \ (t\in \mathbb Z_{+},\ u\in \mathbb R^n)
\end{equation}
where $\mathbb Z$ (respectively, $\mathbb Z_{+}$) is the set of
all entire (respectively, entire nonnegative) numbers and $f\in
C(\mathbb Z\times \mathbb R^n,\mathbb R^n)$, with almost periodic
coefficients.

We show that, in general, the answer to Seifert's question for
(\ref{eqi1}) is negative, even in the scalar case, i.e. when
$n=1$. We prove this claim by constructing an appropriate
counterexample.

Additionally, we prove a positive answer for a special class of difference equations. To that end, we impose stronger assumptions on the right-hand side of our difference equation, and introduce the so--called equations (\ref{eqi1}) with  \emph{weak convergence}, for
which we prove that the response to Seifert's question is affirmative.

We present our results within the framework of general non-autonomous
dynamical systems (cocycles) and we apply our abstract theory already  developed in some previous papers (see, e.g. \cite{CC_2010, CC_I} and the references therein)
to study some classes of difference equations.

The paper is organized as follows.

In Section 2, we briefly recall some notions (global attractor,
minimal set, point/com\-pact dissipativity, non-autonomous dynamical
systems with convergence, Levitan/Bohr almost periodicity, almost
automorphy, recurrence, Poisson stability, etc) and facts from the
theory of dynamical systems which will be necessary in this paper.
We give here also some results concerning a special class of
non-autonomous dynamical systems (NAS): the so-called NAS with
weak convergence.

In Section 3, we analyze Seifert's Problem for non-autonomous difference equations. We
first introduce the notion of discretization of a dynamical
system with continuous time (a flow), and we establish some
relations between the given flow and its discretization.

Next, in Section 3.2 we construct an example of a one-dimensional almost periodic dissipative
difference equation of type (\ref{eqi1}) without almost periodic
solutions, what can be interpreted as a negative answer to our problem under study.

However, we establish a positive response in Section 3.3, where we prove that an almost periodic dissipative
difference equation (\ref{eqi1}) with weak convergence admits a
unique almost periodic solution, and this solution, in general, is
not the unique  solution of this equation which is bounded on
$\mathbb Z$.

Section 3.4 is devoted to the study of uniform compatible solutions of (\ref{eqi1}) by
the character of recurrence (in the sense of B. A. Shcherbakov
\cite{Shch_1972,scher75,Shch_1985}). In this way we obtain some tests for the
existence of periodic (respectively, almost periodic, almost
automorphic, recurrent) solutions of equation (\ref{eqi1}) and
also asymptotically periodic (respectively,
 asymptotically almost periodic,
asymptotically almost automorphic, asymptotically recurrent)
solutions.

Finally, in Section 3.5, we refine these previous results of Section 3.4 for a special class of equations of type (\ref{eqi1}).

\section{Nonautonomous Dynamical Systems with Convergence and/or Weak
Convergence}\label{sec2}

Let us start by recalling some concepts and notation about the theory of
non-autonomous dynamical systems which will be necessary for our analysis. A more detailed analysis can be found, for instance, in \cite{CC_2010, CC_I}.

\subsection{Compact Global Attractors of Dynamical Systems}

Let $(X,\rho)$ be a metric space, $\mathbb{R}$ $(\mathbb{Z})$ be the
group of real (integer) numbers, $\mathbb{R_{+}}$
$(\mathbb{Z_{+}})$ be the semi-group of nonnegative real
(integer) numbers, $\mathbb S$ be one of the two sets $\mathbb{R}$
or $\mathbb{Z}$ and $\mathbb{T}\subseteq\mathbb{S}$
$(\mathbb{S_{+}}\subseteq\mathbb{T}$) be a sub-semigroup of the
additive group $\mathbb{S}$.

A \emph{dynamical system} is a triplet $(X,\mathbb{T},\pi)$, where $\pi:\mathbb{T}\times X\to X$
is a continuous mapping satisfying the following conditions:
\begin{equation}\label{eq1.0.1}
\pi(0,x)=x\ (\forall x\in X);\nonumber
\end{equation}
\begin{equation}\label{eq1.0.2}
\pi(s,\pi(t,x))=\pi(s+t,x)\ (\forall t,\tau \in\mathbb T\
\mbox{and}\ x\in X).\nonumber
\end{equation}

The function $\pi(\cdot,x):\mathbb{T} \to X$ is called a
\emph{motion} passing through the point $x$ at the moment $t=0$
and the set $\Sigma_x:= \pi(\mathbb{T},x)$ is called the
\emph{trajectory} of this motion.

A nonempty set $M\subseteq X$ is called \emph{positively
invariant} (\emph{negatively invariant}, \emph{invariant}) with
respect to the dynamical system $(X,\mathbb{T},\pi)$ or, simply,
positively invariant (negatively invariant, invariant), if
$\pi(t,M)\subseteq M$ $(M\subseteq \pi(t,M), \pi(t,M)=M)$  for
every $t\in \mathbb{T}$.

A closed positively invariant set, which does not contain any own
closed positively invariant subset, is called \emph{minimal}.

It is easy to see that every positively invariant minimal set is
invariant.

Let $M\subseteq X$. The set $$
\omega(M):=\bigcap\limits_{t\ge0}\overline{\bigcup
\limits_{\tau\ge t}\pi(\tau ,M)} $$ is called the
\emph{$\omega$-limit} of $M$.

The dynamical system $(X,\mathbb{T},\pi)$ is called:
\begin{list}{$-$}{}
\item \emph{point dissipative} if there exists a nonempty compact
subset $K\subseteq X$ such that for every $x\in X$
\begin{equation}\label{eq1.2.1}
\lim \limits_{t\to+\infty}\rho(\pi(t,x),K)=0;
\end{equation}
\item \emph{compact dissipative} if the equality (\ref{eq1.2.1})
takes place uniformly with respect to $x$ in any compact subset of $X$.
\end{list}

Let $(X,\mathbb T, \pi)$ be compact dissipative and $K$ be a
compact set attracting every compact subset from $X$. Let us set
\begin{equation}\label{eq1.2*}
 J_X:= \omega (K):= \bigcap \limits
_{t \ge 0} \overline { \bigcup \limits_{\tau \ge t} \pi({\tau},K) }.
\end{equation}

It can be shown \cite[Ch.I]{Che_2004} that the set $J_X$ defined by
equality (\ref{eq1.2*}) does not depend on the choice of the
attracting set $K$, but is characterized only by the properties of the
dynamical system $(X,\mathbb T, \pi)$ itself. The set $J_X$ is
called the \emph{Levinson center} of the compact dissipative
dynamical system $(X,\mathbb T, \pi)$.

More generally, a compact invariant set $J\subset X$ is called the Levinson
center of the compact dissipative dynamical system $(X,\mathbb T, \pi)$ if $J$
attracts every compact subset of $X$, which means that
\begin{equation*}
\lim \limits_{t\to+\infty}\rho(\pi(t,x),J)=0,
\end{equation*}
uniformly with respect to $x\in M$, and for all compact subset $M$ of $X$. It is
worth noticing that  this concept does not coincide, in general, with that of global
attractor (since the latter attracts the bounded subsets of $X$). For a more detailed
analysis on the relationship between these two concepts, see Cheban \cite{Che_2004}.

\subsection{Global attractor of cocycles}

Let $\mathbb T_{1}\subseteq
\mathbb T_{2}$ be two sub-semigroups of the group $\mathbb S$
($\mathbb S_{+}\subseteq \mathbb T_{1}$).

A triplet $\langle (X,\mathbb{T}_1,\pi),\,(Y,\mathbb{T}_2,\sigma),
\,h\rangle $, where $h$ is a homomorphism from
$(X,\mathbb{T}_1,\pi)$ onto $(Y,\mathbb{T}_2,\sigma)$ (i.e., $h$
is continuous and $h(\pi(t,x))=\sigma(t,h(x))$ for all
$t\in\mathbb T_{1}$ and $x\in X$), is called a
\textit{non-autonomous dynamical system}.

Let $(Y,\mathbb{T}_2,\sigma)$ be a dynamical system, $W$
a complete metric space, and $\varphi$ a continuous mapping from
$\mathbb{T}_1\times W\times Y$ into $W$, possessing the following
properties:
\begin{enumerate}
\item[a.] $\varphi(0,u,y)=u$ $(u\in W, y\in Y)$; \item[b.]
$\varphi(t+\tau,u,y)= \varphi(\tau,\varphi(t,u,y),\sigma(t,y))$
$(t,\tau\in\mathbb{T}_1,\, u\in W, y\in Y).$
\end{enumerate}
Then, the triplet $\langle W, \varphi,
(Y,\mathbb{T}_2,\sigma)\rangle $ (or shortly $\varphi$) is called
\cite{Sell71} \textit{a cocycle }on $(Y,\mathbb{T}_2,\sigma)$ with
fiber $W$.

Let $X:= W\times Y$ and let us define a mapping $\pi: X\times
\mathbb{T}_1\to X$ as follows:
$\pi((u,y),t):=(\varphi(t,u,y),\sigma(t,y))$ (i.e.,
$\pi=(\varphi,\sigma)$). Then, it is easy to see that
$(X,\mathbb{T}_1,\pi)$ is a dynamical system on $X$, which is
called a \textit{skew-product dynamical system} \cite{Sell71} and
$h=pr_2:X\to Y$ is a homomorphism from $(X,\mathbb{T}_1,\pi)$ onto
$(Y,\mathbb{T}_2,\sigma)$ and, hence, $\langle
(X,\mathbb{T}_1,\pi),\, (Y,\mathbb{T}_2,\sigma), h\rangle $ is a
non-autonomous dynamical system. Note that $pr_1$ and $pr_2$ denote the projection
mappings with respect to the first and second variables, i.e., $pr_1(w,y)=w$, $ pr_2(w,y)=y$ for $(w,y)\in X$.

Thus, if we have a cocycle $\langle W, \varphi, (Y,\mathbb{T}_2,
\sigma)\rangle $ on the dynamical system $(Y,\mathbb{T}_2,\sigma)$
with fiber $W$, then it generates a non-autonomous dynamical
system $\langle (X,\mathbb{T}_1,\pi),$\ $(Y,\mathbb{T}_2,\sigma),
h\rangle $ ($X:= W\times Y$), called non-autonomous dynamical
system generated by the cocycle $\langle W, \varphi,
(Y,\mathbb{T}_2,\sigma)\rangle $ on $(Y,\mathbb{T}_2,\sigma)$.

Non-autonomous dynamical systems (cocycles) play a very important
role in the study of non-autonomous evolutionary differential/difference
equations. Under appropriate assumptions, every non-autonomous
differential/difference equation generates a cocycle (a non-autonomous
dynamical system). Several examples can be found, for instance, in \cite{CC_2010}.

A family $ \{ I_{y} \ \vert \   y \in Y \} \ ( I_{y} \subset W
)$ of nonempty compact subsets of $W$ is called (see, for example,
\cite{Che_2004}) \emph{a compact pullback attractor}
(\emph{uniform pullback attractor}) of the cocycle $ \varphi $, if
the following conditions hold:
\begin{enumerate}
\item the set $ I := \bigcup \{ I_{y} \ \vert \   y \in Y \} $ is
relatively compact; \item the family $ \{ I_{y} \ \vert \   y \in
Y \} $ is invariant with respect to the cocycle $ \varphi $, i.e.,
$ \varphi (t, I_{y}, y ) = I_{\sigma (t,y )} $ for all $ t \in
\mathbb T_{+} $ and $ y \in Y $; \item for all $ y \in Y $
(uniformly in $y \in Y$) and $ K \in \mathcal{C}(W) $
$$
\lim \limits _{t \to + \infty } \beta (\varphi (t, K, \sigma(-t,y)
), I_{y}) =0 ,
$$
where $ \beta (A,B): = \sup \{ \rho (a,B) : a \in  A \} $ is the Hausdorff
semi-distance, and $\mathcal{C}(W)$ denotes de family of compact subsets of $W$.
\end{enumerate}

Below in this subsection we suppose that $\mathbb T_2=\mathbb S.$

A family $ \{ I_{y} \ \vert \   y \in Y \} ( I_{y} \subset W ) $
of nonempty compact subsets is called a \emph{compact global attractor}
of the cocycle\index{global attractor of the cocycle} $\varphi,$
if the following conditions are fulfilled:
\begin{enumerate}
\item the set $ I := \bigcup \{ I_{y} \ \vert \   y \in Y \} $ is
relatively compact; \item the family $ \{ I_{y} \ \vert \   y \in
Y \} $ is invariant with respect to the cocycle $ \varphi $; \item
the equality
$$
\lim \limits _{t \to + \infty } \sup \limits _{y \in Y } \beta
(\varphi (t,K,y),I)=0
$$
holds for every $K\in \mathcal{C}(W)$.
\end{enumerate}

Let $M \subseteq W$ and
\begin{equation}\label{eq2.7.3}
\omega_y(M):=\bigcap_{t\ge 0} \overline{\bigcup_{\tau\ge t}
\varphi(\tau,M,\sigma(-\tau,y))}\nonumber
\end{equation}
for all $y \in Y$.

A cocycle $\varphi $ over $(Y,\mathbb S,\sigma)$ with fiber
$W$ is said to be \emph{compact dissipative}, if there exits a nonempty
compact $K \subseteq W$ such that
\begin{equation}\label{eq2.7.8}
\lim_{t \to + \infty} \sup \{ \beta (\varphi(t,M,y),K) \ \vert \  y \in Y
\}=0
\end{equation}
for any  $M \in \mathcal{C}(W)$.

Recall that a function $F\in C(\mathbb T,\mathbb R)$ is said
\emph{to possess the $(S)$-property  (for example, periodicity,
almost periodicity, recurrence, asymptotically almost periodicity
and so on)}, if the motion $\sigma(\tau,F),$ generated by the
function $F$ in the shift dynamical system $(C(\mathbb T,\mathbb
R),\mathbb T,\sigma)$, possesses this property (see \cite[Ch. II]{Che_2004} for more details).

Then, we can now establish the following result which will be used in the proof of our main results in this paper.

\begin{theorem}\label{thCA}\cite[ChII]{Che_2004}
Let $Y$ be compact, $\langle W,\varphi,(Y,\mathbb S,\sigma)\rangle
$  compact dissipative, and $K$  the nonempty compact subset
of $W$ appearing in (\ref{eq2.7.8}). Then:
\begin{enumerate}
\item[1.] $I_y=\omega_y(K) \ne \emptyset $, is compact, $ I_y
\subseteq K$ and
$$
\lim_{t \to + \infty} \beta(\varphi(t,K,\sigma(-t,y)),I_y)=0
$$
for every $y \in Y$; \item[2.] $\varphi(t,I_y,y)=I_{\sigma(t,y)} $ for
all $y \in Y$ and $t \in \mathbb S_+$; \item[3.]
\begin{equation}\label{eq2.7.10}
\lim_{t \to + \infty}\beta(\varphi(t,M,\sigma(-t,y)),I_y)=0\nonumber
\end{equation}
for all $M \in \mathcal{C}(W)$ and $ y \in Y$ ; \item[4.]
\begin{equation}\label{eq2.7.11}
\lim _{t \to + \infty} \sup \{ \beta(\varphi(t,M,\sigma(-t,y)),I)\ \vert
\
 y \in Y \}=0\nonumber
\end{equation}
for any $M \in \mathcal{C}(W)$, where $I:=\cup \{ I_y \ \vert \ y \in Y \}
$; \item[5.] $ I_y=pr_1J_y$ for all $ y \in Y$, where $J$ is the
Levinson center of $(X,\mathbb T_+,\pi)$, and hence $I=pr_1J$;
\item[6.] the set $I$ is compact; \item[7.] the set $I$ is
connected if one of the next two conditions is fulfilled:
\begin{enumerate}
\item $\mathbb S_+=\mathbb R_+$ and the spaces $W$ and $Y$ are
connected; \item $\mathbb S_+=\mathbb Z_+$ and the space $W \times
Y$ possesses the $(S)$-property or it is connected and locally
connected.
\end{enumerate}
\end{enumerate}
\end{theorem}

\subsection{Non-Autonomous Dynamical Systems with Convergence and Weak Convergence}

First, let us recall (see \cite{Che_2004}) that a non-autonomous dynamical system
$\langle (X,
$ $\mathbb T_{1},$ $ \pi ),$ $ (Y,$ $ \mathbb T_{2},$ $\sigma ),$ $ h
\rangle $ is said to be \emph{convergent} if the following
conditions are fulfilled:
\begin{enumerate}
\item the dynamical systems $ (X, \mathbb T_{1} , \pi ) $ and $
(Y, \mathbb T_{2} ,\sigma ) $ are compact dissipative; \item the
set $ J_{X} \bigcap X_{y} $ contains no more than one point for
all $ y \in J_{Y} $, where $ X_{y} := h^{-1}(y):= \{ x \vert  x
\in X, h(x)=y \} $ and $ J_{X}$ (respectively, $J_{Y}$) is the
Levinson center of the dynamical system $ (X, \mathbb T_{1} , \pi
)$ (respectively, $(Y, \mathbb T_{2} ,\sigma ))$.
\end{enumerate}

Thus, a non-autonomous dynamical system $\langle(X,\mathbb
T_1,\pi),(Y,\mathbb T_2,\sigma),h\rangle$ is convergent, if the
systems $(X,\mathbb T_1,\pi)$ and $(Y,\mathbb T_2,\sigma)$ are
compact dissipative with Levinson centers $J_X$ and $J_Y$
respectively, and $J_X$ has  ``trivial" sections, i.e.,
$J_X\bigcap X_y$ consists of a single point for all $y\in J_Y$. In
this case, the Levinson center $J_X$ of the dynamical system
$(X,\mathbb T_1,\pi)$ is a copy (an homeomorphic image) of the
Levinson center $J_Y$ of the dynamical system $(Y,\mathbb
T_2,\sigma)$. Thus, the dynamics on $J_X$ is the same as on $J_Y$.

Before introducing the class of non-autonomous dynamical systems with weak convergence,
let us recall some definition concerning the different recurrence properties of
points and motions (see \cite{Che_2009} for more details).

Let $(X,\mathbb{T}, \pi)$ be a dynamical system.
Given $\varepsilon>0$, a number $\tau\in\mathbb{T}$ is called an
$\varepsilon-$\emph{shift} (respectively, an
$\varepsilon-$\emph{almost period}) of the point $x\in X$, if
$\rho(\pi(\tau,x),x)<\varepsilon$ (respectively, $\rho
(\pi(\tau+t,x),\pi(t,x))<\varepsilon$ for all $t\in\mathbb{T}$).

A point $x\in X$ is called \emph{almost recurrent} (respectively,
\emph{Bohr almost periodic}), if for any $\varepsilon>0$ there
exists a positive number $l$ such that in any segment of length
$l$ there is an $\varepsilon-$shift (respectively,
$\varepsilon-$almost period) of the point $x\in X$.

If the point $x\in X$ is almost recurrent and the set
$H(x):=\overline {\{\pi(t,x)\ \vert\ t\in\mathbb{T}\}}$ is
compact, then $x$ is called \emph{recurrent}, where by bar we
denote the closure in $X$.

Denote by $\mathfrak{N}_{x}:=\{\{t_{n}\}\subset\mathbb{T}%
\ :\ \mbox{such that}\ \{\pi(t_{n},x)\}\to x \ \mbox{and}\ \{t_{n}\}\to
\infty\}.$

A point $x\in X$ of the dynamical system $(X,\mathbb{T}, \pi)$ is
called \emph{Levitan almost periodic} \cite{Lev-Zhi}, if there
exists a dynamical system $(Y,\mathbb{T},\lambda)$ and a Bohr
almost periodic point $y\in Y$ such that
$\mathfrak{N}_{y}\subseteq\mathfrak{N}_{x}.$

A point $x\in X$ is called \emph{stable in the sense of Lagrange}
$(st.$L for short), if its trajectory $\{\pi(t,x)\ :\
t\in\mathbb{T}\}$ is relatively compact.

A point $x\in X$ is called \emph{almost automorphic}
\cite{Lev-Zhi} for the dynamical system $(X,\mathbb{T},\pi),$
if the following conditions hold:

\begin{enumerate}
\item $x$ is st.$L$;

\item there exists a dynamical system $(Y,\mathbb{T},\lambda),$ a homomorphism
$h$ from $(X,\mathbb{T},\pi)$ onto $(Y,\mathbb{T},\lambda),$ and an almost
periodic (in the sense of Bohr) point $y\in Y$ such that $h^{-1}(y)=\{x\}.$
\end{enumerate}

\begin{remark}
\label{r4.7}1. Every almost automorphic point $x\in X$ is also Levitan almost periodic.

2. A Levitan almost periodic point $x$ with relatively compact trajectory
$\{\pi(t,x):\ t\in T\}$ is also almost automorphic. In other
words, a Levitan almost periodic point $x$ is almost automorphic, if and only
if its trajectory $\{\pi(t,x):\ t\in T\}$ is relatively compact.

3. Let $(X,T,\pi)$ and $(Y,T,\lambda)$ be two dynamical systems, $x\in X$ and
the following conditions be fulfilled:

\begin{enumerate}
\item a point $y\in Y$ is Levitan almost periodic;

\item $\mathfrak{N}_{y}\subseteq \mathfrak{N}_{x}$.
\end{enumerate}

Then, the point $x$ is also Levitan almost periodic.

4. Let $x\in X$ be a st.$L$ point, $y\in Y$ be an almost
automorphic point and $\mathfrak{N}_{y}\subseteq
\mathfrak{N}_{x}$. Then, the point $x$ is almost automorphic too.
\end{remark}

Recall \cite{Che_2009} that a point $x\in X$ is called
\emph{asymptotically $\tau$--periodic} (respectively,
 \emph{asymptotically Bohr
almost periodic}, \emph{asymptotically recurrent}), if there exists
a $\tau$-periodic (respectively, Bohr almost periodic, recurrent)
point $p\in X$ such that $\lim\limits_{t\to
+\infty}\rho(\pi(t,x),\pi(t,p))=0$.

Now, we introduce the concept of non-autonomous dynamical
systems with weak convergence, which is very close to convergent systems, but possessing
a non-trivial global attractor. This means that this class of
non-autonomous systems will conserve almost all properties of
convergent systems, but will have a ``nontrivial" global attractor
$J_X$, i.e., there exists at least one point $y\in J_{Y}$ such
that the set $J_{X}\bigcap X_{y}$ contains more than one point.

A non-autonomous dynamical system $\langle(X,\mathbb
T_1,\pi),(Y,\mathbb T_2,\sigma),h\rangle$ is said to be \emph{weak
convergent}, if the following conditions hold:
\begin{enumerate}
\item the dynamical systems $(X,\mathbb T_1,\pi)$ and $(Y,\mathbb T_2,\sigma)$
are compact dissipative with Levinson centers $J_X$ and $J_Y$
respectively;
\item it follows that
\begin{equation}\label{eqw1}
\lim\limits_{t\to +\infty}\rho(\pi(t,x_1),\pi(t,x_2))=0,\nonumber
\end{equation}
for all $x_1,x_2 \in J_X$ with $h(x_1)=h(x_2)$.
\end{enumerate}

\begin{remark}
\emph{It is clear that every convergent non-autonomous dynamical
system is weak convergent. The inverse statement, generally
speaking, is not true. See \cite{CC_I} for a counterexample
confirming this statement.}
\end{remark}

\section{Analysis of Seifert's problem for dissipative difference equations}\label{sec5}

We can now analyze the questions related to Seifert's problem mentioned in the Introduction of this paper.
To be more precise, consider again the initial differential equation (\ref{eqI1*})
\begin{equation*}
x'=f(t,x),
\end{equation*}
where $f\in C(\mathbb R\times \mathbb R^n,\mathbb R^n)$. Assume
that the right-hand side of (\ref{eqI1*}) satisfies hypotheses
ensuring the existence, uniqueness and extendability of solutions
of (\ref{eqI1*}),  i.e., for all $(t_0,x_0)\in \mathbb R\times
\mathbb R^n$ there exists a unique solution $x(t;t_0,x_0)$ of
(\ref{eqI1*}) with initial data $t_0,x_0$, and defined for all
$t\geq t_0$.

Equation (\ref{eqI1*}) (respectively, the function $f$) is called
\emph{regular}, if for all $x_0\in \mathbb R^n$ and $g\in
H(f):=\overline{\{f_{\tau}:\ \tau\in \mathbb R\}}$ (where the bar
denotes the closure in the space $C(\mathbb R\times \mathbb
R^n,\mathbb R^n)$ and $f_{\tau}(t,x):=f(t+\tau,x)$ for all
$(t,x)\in\mathbb R\times \mathbb R^n$) the equation
\begin{equation}\label{eqI2}
x'=g(t,x)\nonumber
\end{equation}
has a unique solution $\varphi(t,x,g)$ passing through the point
$x_0$ at the initial moment $t=0$, and defined on $\mathbb
R_{+}:=\{t\in\mathbb R |\ t\ge 0\}$.

Then, the following result holds.

\begin{theorem}\label{thI1}\cite[ChII]{Che_2004} Suppose that $f\in C(\mathbb R\times \mathbb R^n,\mathbb
R^n)$ is regular and $H(f)$ is a compact subset of $C(\mathbb
R\times \mathbb R^n,\mathbb R^n)$. Then, the following statements
are equivalent:
\begin{enumerate}
\item equation (\ref{eqI1*}) is uniformly dissipative; \item there
exists a positive number $R_0$ such that
\begin{equation}\label{eqI3}
\limsup\limits_{t\to +\infty}|\varphi(t,x,g)|\le R_0
\end{equation}
for all $(x,g)\in\mathbb R^n\times H(f)$.
\end{enumerate}
\end{theorem}

At light of Theorem \ref{thI1}, it is said that equation
(\ref{eqI1*}) is \emph{dissipative} (in fact the family of equations
(\ref{eqI3}) is collectively dissipative, but we use this shorter
terminology) if (\ref{eqI3}) holds.

\begin{remark}\label{remS1} Let $\langle \mathbb R^n,\varphi, (Y,\mathbb
R,\sigma)\rangle$ ($Y=H(f):=\overline{\{f_{\tau}:\ \tau\in\mathbb
R\}}$ and $(Y,\mathbb R,\sigma)$ is the shift dynamical system on
$Y$) be the cocycle generated by the differential equation
(\ref{eqI1*}). If $H(f)$ is a compact subset of $C(\mathbb
R\times \mathbb R^n,\mathbb R^n))$, then (\ref{eqI1*}) is
dissipative if and only if the cocycle $\varphi$, generated by
(\ref{eqI1*}), is compact dissipative.
\end{remark}

Now, before establishing and proving our main results about Seifert's problem in the
case of difference equations (or discrete non-autonomous dynamical systems),
we need some results on discretization which we will consider in the next subsection.

\subsection{Discretization of NAS with continuous time}\label{sec4}

Let $\mathbb T=\mathbb R_{+}$ or $\mathbb R$. Consider a
non-autonomous dynamical system $\langle (X,\mathbb
T,\tilde{\pi}),$ $ (Y,\mathbb T,\tilde{\sigma}),$ $h \rangle$ with
continuous time $\mathbb T$ and denote by $\mathcal T:=\mathbb T
\bigcap \mathbb Z$.

The non-autonomous dynamical system $\langle (X,\mathcal T,\pi),
(Y,\mathcal T,\sigma),h \rangle$ with discrete time $\mathcal T$
is called the discretization of $\langle (X,\mathbb
T,\tilde{\pi}), (Y,\mathbb T,\tilde{\sigma}),h \rangle$, if the
following conditions are fulfilled:
\begin{enumerate}
\item $\pi(t,x)=\tilde{\pi}(t,x)$ for all $(t,x)\in \mathcal
T\times X$; \item $\sigma(t,x)=\tilde{\sigma}(t,x)$ for all
$(t,x)\in \mathcal T\times Y$.
\end{enumerate}

\begin{lemma}\label{lD1} Let $(X,\mathbb R_{+},\tilde{\pi})$ be an
autonomous dynamical system with continuous time $\mathbb R_{+}$
and $(X,\mathbb Z_{+},\pi)$ be its discretization. Then, the
following statements hold:
\begin{enumerate}
\item If $\tilde{\gamma}:\mathbb R\mapsto X$ is an entire
trajectory of $(X,\mathbb R_{+},\tilde{\pi})$, then the mapping
$\gamma :\mathbb Z\to X$ defined by the equality
$\gamma(n):=\tilde{\gamma}(n)$ (for all $n\in\mathbb Z$) is an
entire trajectory of $(X,\mathbb Z_{+},\pi)$; \item If
$\gamma:\mathbb Z\mapsto X$ is an entire trajectory of $(X,\mathbb
Z_{+},\pi)$, then the mapping $\tilde{\gamma} :\mathbb R\mapsto X$
defined by the equality
\begin{equation}\label{eqD1}
\tilde{\gamma}(t):=\tilde{\pi}(\{t\},\gamma([t]))\ (\forall\ \
t\in\mathbb R)
\end{equation}
is an entire trajectory of $(X,\mathbb R_{+},\tilde{\pi}),$ where
$[t]\in\mathbb Z$ is the entire part of the number $t$ and
$\{t\}\in [0,1)$ is its fractional part.
\end{enumerate}
\end{lemma}
\begin{proof} The first statement of Lemma is trivial. To prove the second statement we need to verify
that the equality
\begin{equation}\label{eqD0}
\tilde{\gamma}(t+\tau)=\tilde{\pi}(t,\tilde{\gamma}(\tau))
\end{equation}
holds for all $t\in\mathbb R_{+}$ and $\tau\in \mathbb R$. Let now $t\in
\mathbb R_{+}$, $\tau\in\mathbb R$ and $\tilde{\gamma}:\mathbb
R\to X$ be the mapping defined by (\ref{eqD1}), then we have
\begin{equation}\label{eqD2}
\tilde{\pi}(t,\tilde{\gamma}(\tau))=\tilde{\pi}([t]+\{t\},\tilde{\gamma}([\tau]+\{\tau\}))=\tilde{\pi}(\{t\}+\{\tau\},\gamma([t]+[\tau]).
\end{equation}
Logically, two cases are possible:

1. $\{t\}+\{\tau\}\in [0,1)$: Then, from (\ref{eqD2}) it follows
(\ref{eqD0}).

2. $\{t\}+\{\tau\}=1+r,$ where $r\in [0,1)$ and, consequently,
$r=\{t+\tau\}$ and $[t+\tau]=[t]+[\tau]+1$: Then, from (\ref{eqD2})
we have
\begin{eqnarray}\label{eqD3}
\tilde{\pi}(t,\tilde{\gamma}(\tau))&=&\tilde{\pi}(\{t\}+\{\tau\},\gamma([t]+[\tau])\\
&=&\tilde{\pi}(r+1,\gamma([t]+[\tau]))\nonumber
\\
&=& \tilde{\pi}(r,\gamma(1+[t]+[\tau]))\notag\\
&=&\tilde{\gamma}(t+\tau).\notag
\end{eqnarray}
Lemma is proved.
\end{proof}

Let $(X,\mathbb R_{+},\tilde{\pi})$ be a semi-flow and
$\tilde{\gamma}$ be an entire trajectory, then $\gamma :\mathbb
Z\mapsto X$, defined by the equality
$\gamma(n):=\tilde{\gamma}(n)$ for all $n\in\mathbb Z$ is called
\emph{discretization} of $\tilde{\gamma}$.

\begin{lemma}\label{lD2} Suppose that the following conditions are fulfilled:
\begin{enumerate}
\item $(X,\mathbb R_{+},\tilde{\pi})$ is an autonomous dynamical
system with continuous time $\mathbb R_{+}$ and $(X,\mathbb
Z_{+},\pi)$ is its discretization; \item $\tilde{\gamma}:\mathbb
R\mapsto X$ is an entire trajectory of $(X,\mathbb
R_{+},\tilde{\pi})$ and $\gamma$ is its discretization.
\end{enumerate}

Then the trajectory $\tilde{\gamma}$ is almost periodic with
respect to $(X,\mathbb R_{+},\tilde{\pi})$ if and only if $\gamma$
is almost periodic with respect to $(X,\mathbb Z_{+},\pi)$.
\end{lemma}
\begin{proof} Let $\tilde{\gamma}$ be an almost periodic motion of
$(X,\mathbb R_{+},\tilde{\pi}))$ and $\{n_k^{'}\}\subseteq \mathbb
Z$ be an arbitrary sequence of entire numbers. Consider the
sequence $\{\gamma(n+n_k^{'})\}_{n\in\mathbb Z}$. Since the
trajectory $\tilde{\gamma}$ is almost periodic with respect to
$(X,\mathbb R_{+},\tilde{\pi})$, then from the functional sequence
$\{\tilde{\gamma}(t+n_k^{'})\}$ ($t\in\mathbb R$) we can extract a
subsequence $\{\tilde{\gamma}(t+n_k)\}$ ($t\in\mathbb R$) which is
uniformly convergent  with respect to $t\in\mathbb R$. In
particular we have
\begin{equation}\label{eqD4}
\sup\limits_{t\in\mathbb
R}\rho(\tilde{\gamma}(t+n_{k}),\tilde{\gamma}(t+n_{l}))\to 0
\end{equation}
as $k,l\to +\infty$. Taking into account (\ref{eqD4}) we have
\begin{eqnarray}\label{eqD5}
 \sup\limits_{t\in\mathbb
Z}\rho(\gamma(t+n_{k}),\gamma(t+n_{l}))&=&\sup\limits_{t\in\mathbb
Z}\rho(\tilde{\gamma}(t+n_{k}),\tilde{\gamma}(t+n_{l})) \\
&\le& \sup\limits_{t\in\mathbb
R}\rho(\tilde{\gamma}(t+n_{k}),\tilde{\gamma}(t+n_{l}))\to 0  \nonumber
\end{eqnarray}
as $k,l\to +\infty$. Since the space $X$ is complete, then the
sequence $\{\gamma(n+n_k^{'})\}_{n\in\mathbb Z}$ converges
uniformly with respect to $n\in\mathbb Z$ and, consequently,
$\gamma :\mathbb Z\mapsto X$ is almost periodic.

As for the converse statement, let $\gamma :\mathbb Z\mapsto X$ be
an almost periodic motion of $(X,\mathbb Z_{+},\pi)$ and
$\varepsilon >0$  an arbitrary positive number. We choose a
number $\delta =\delta (\varepsilon)>0$ from the integral
continuity of the dynamical system $(X,\mathbb
R_{+},\tilde{\pi})$, i.e., such that $\rho(x_1,x_2)<\delta$
implies $\rho(\tilde{\pi}(t,x_1),\tilde{\pi}(t,x_2))<\varepsilon$,
for all $t\in [0,1]$. For the number $\delta$, we can choose a
relatively dense subset $\mathcal P_{\delta}\subseteq \mathbb Z$
such that
\begin{equation}\label{eqD6}
\rho(\gamma(n+\tau),\gamma(n))<\delta
\end{equation}
for all $n\in\mathbb Z$ and $\tau\in\mathcal P_{\delta}$. Then we
have
\begin{equation}\label{eqD7}
\rho(\tilde{\gamma}(t+\tau),\tilde{\gamma}(t))=\rho(\tilde{\pi}(\{t\},
\gamma([t]+\tau)),\tilde{\pi}(\{t\},\gamma([t]))<\varepsilon
\end{equation}
because $\rho(\gamma([t]+\tau),\gamma([t])<\delta$ for all
$t\in\mathbb R$ and $\tau\in\mathcal P_{\delta}\subseteq \mathbb
Z\subseteq \mathbb R$. Thereby, for arbitrary $\varepsilon >0$ we
find a relatively dense subset $\mathcal
P_{\delta(\varepsilon)}\subseteq \mathbb R$ of
$\varepsilon$--almost periods of the motion $\tilde{\gamma}$. This
means its almost periodicity, and the lemma is proved.
\end{proof}

\begin{remark}\label{rem11} Note that Lemma \ref{lD2} is close to
Theorem 1.27 \cite[ChI,p.47]{Cor}, but does not follow from this
statement.
\end{remark}

\begin{coro}\label{corD1} Let $(X,\mathbb R_{+},\tilde{\pi})$ be a
semi-flow without almost periodic motions. Then, its discretization
$(X,\mathbb Z_{+},\pi)$ does not possess any almost periodic motion.
\end{coro}

\subsection{A negative answer: One-dimensional almost periodic dissipative difference
equation without almost periodic solutions}

Now we investigate the following interesting question.

\textbf{Problem.}(\emph{Seifert's problem for almost periodic
dynamical systems}) Suppose that the following conditions are
fulfilled:
\begin{enumerate}
\item $\langle W,\varphi, (Y,\mathbb T,\sigma)\rangle$ is a
compact dissipative cocycle with continuous ($\mathbb{T}=\mathbb
R_{+}$ or $\mathbb R$) or discrete ($\mathbb{T}=\mathbb Z_{+}$ or
$\mathbb Z$) time; \item $(Y,\mathbb T,\pi)$ is an almost periodic
minimal set.
\end{enumerate}

Does the skew-product dynamical system $(X,\mathbb T_{+},\pi)$,
generated by the cocycle $\varphi$ ($X=W\times Y$,
$\pi=(\varphi,\sigma)$ and $\mathbb T_{+}:=\{t\in \mathbb T:\ t\ge
0\}$) possess an almost periodic motion?

Fink and Fredericson \cite{FF_1971} and Zhikov \cite{Zhi_1972}
established that, in general,  even when (\ref{eqI1*}) is scalar,
the answer to Seifert's question is negative. Namely, in these
works they constructed a differential equation
\begin{equation}\label{eqSP1}
x'=f(\tilde{\sigma}(t,y),x),\ \ (x\in\mathbb R,\ y\in \mathcal
T^{2})
\end{equation}
where $\mathcal T^{2}$ is a two-dimensional torus, $(\mathcal
T^{2},\mathbb R,\tilde{\sigma})$ is an almost periodic minimal
dynamical system (in fact this is an irrational winding of a
two-dimensional torus $\mathcal T^{2}$) and $f\in C(\mathcal
T^2\times \mathbb R,\mathbb R)$ with the following properties:
\begin{enumerate}
\item the cocycle generated by (\ref{eqSP1}) is dissipative; \item
the skew-product dynamical system generated by (\ref{eqSP1}) does
not possess any almost periodic motions.
\end{enumerate}

Below we will prove that there exists a one-dimensional
($W=\mathbb R$) almost periodic compact dissipative cocycle
$\varphi$ with discrete time for which the corresponding
skew-product dynamical system $(X,\mathbb Z_{+},\pi)$ does not
have any almost periodic motions. To this end we denote by
$\langle \mathbb R,\tilde{\varphi},(\mathcal T^2,\mathbb
R,\tilde{\sigma})\rangle$ the cocycle generated by (\ref{eqSP1}).
Then we have
\begin{equation}\label{eqSP2}
\tilde{\varphi}(t,u,y)=u+\int_{0}^{t}f(\tilde{\sigma}(\tau,y),\tilde{\varphi}(\tau,u,y))d\tau.
\end{equation}
Denote by $(\mathcal T^2,\mathbb Z,\sigma)$ the discretization of the
dynamical system $(\mathcal T^2,\mathbb Z,\sigma)$, and let $y_0\in
\mathcal T^2$ be an arbitrary point. Then, by Lemma \ref{lD2}, the
trajectory of the point $y_0$ is almost periodic and,
consequently, the set $Y=H(y_0):=\overline{\{\sigma(n,y_0):\ n\in
\mathbb Z\}}$ is a compact minimal set of the dynamical system
$(\mathcal T^2,\mathbb Z,\sigma))$ consisting of  almost
periodic motions. Now, we define a mapping $\varphi:\mathbb Z\times
\mathbb R\times Y\to \mathbb R $ by the equality $
\varphi(n,u,y):=\tilde{\varphi}(n,u,y)$ and, consequently, we
obtain
\begin{equation}\label{eqSP3}
\varphi(n,u,y)=
u+\int_{0}^{1}f(\tilde{\sigma}(\tau,\sigma(n-1,y),\tilde{\varphi}(\tau,\varphi(n-1,u,y),\sigma(n-1,y)))d\tau.
\end{equation}
Thus, $\varphi(n,u,y)$ is a solution of the difference equation
\begin{equation}\label{eqSP4}
y(t+1)=F(\sigma(t,y),y(t)), \ \ (t\in \mathbb Z,\ y\in Y=H(y_0))
\end{equation}
where $F\in C(Y\times \mathbb R,\mathbb R)$ is defined by
\begin{equation}\label{eqSP5}
F(y,u):=u+\int_{0}^{1}f(\tilde{\sigma}(\tau,y),\tilde{\varphi}(\tau,u,y))d\tau
.
\end{equation}
Since the cocycle $\tilde{\varphi}$, generated by equation
(\ref{eqSP1}), is dissipative by its construction, then the
cocycle $\langle \mathbb R,\varphi,(Y,\mathbb Z,\sigma)\rangle$ is
also dissipative. Now, consider the non-autonomous dynamical system
$\langle (X,\mathbb Z,\pi),(Y,\mathbb Z,\sigma),h\rangle$
generated by the cocycle $\varphi$ (i.e., $X:=\mathbb R\times Y,$ $\pi
:=(\varphi,\sigma)$ and $h=pr_2:X\to Y$). Note that the
dynamical system $(X,\mathbb Z,\pi)$ does not possess any almost
periodic motion. In fact, if we suppose that it is not true, then
there exists a point $(\bar{u},\bar{y})\in X=\mathbb R\times Y$
such that the motion
$(\varphi(n,\bar{u},\bar{y}),\sigma(n,\bar{y}))$ is almost periodic
with respect to the dynamical system $(X,\mathbb Z,\pi)$ and,
consequently, with respect to the dynamical system
$(\tilde{X},\mathbb Z,\sigma)$ (where $\tilde{X}:=\mathbb R\times
\mathcal T^2$) because $(X,\mathbb Z,\pi)$ is a subsystem of
$(\tilde{X},\mathbb Z,\sigma)$. Finally, we note that
$(\tilde{X},\mathbb Z,\sigma))$ is the discretization of the
skew-product dynamical system generated by (\ref{eqSP1}). The
obtained contradiction proves our statement. Thus (\ref{eqSP4}) is
a one-dimensional almost periodic dissipative difference equation
without almost periodic solutions.

\subsection{A positive answer: Almost periodic solutions of almost periodic dissipative
difference equations}\label{sec6}

Let $(Y,\mathbb Z_{+},\sigma)$ be a dynamical system on the metric
space $Y$. In this subsection we prove a positive answer to Seifert's
Problem assuming additional assumptions on our difference equations. First, we suppose that $Y$ is a compact
space. Consider a difference equation
\begin{equation}\label{eqA1}
u(t+1)=f(\sigma(t,y),u(t)) \ \ \ (t\in\mathbb Z_{+},\ y\in Y),
\end{equation}
where $f\in C(Y\times \mathbb R^n,\mathbb R^n)$.

Similarly as for differential equations, the non-autonomous difference equation (\ref{eqA1}) is said to be \cite{Che_2004}
\emph{dissipative}, if there exists a positive number $r$ such
that
\begin{equation}\label{eqA2*}
\limsup\limits_{t\to +\infty}|\varphi(t,u,y)|<r \nonumber
\end{equation}
for all $u\in \mathbb R^n$ and $y\in Y$, where $|\cdot|$ is a norm
on $\mathbb R^n$, and $\varphi(t,u,y)$ is the unique solution of
(\ref{eqA1}) passing through $u\in \mathbb R^n$ at the initial
moment $t=0$.

Below we provide a simple geometric
condition which guarantees existence of a unique almost periodic
solution, and this solution, generally speaking, is not the unique
solution of (\ref{eqA1}) which is bounded on $\mathbb Z$.

\begin{theorem}\label{thA1} Suppose that the following conditions
are fulfilled:
\begin{enumerate}
\item equation (\ref{eqA1}) is dissipative; \item the space $Y$ is
compact, and the dynamical system $(Y,\mathbb Z_{+},\sigma)$ is
minimal; \item for all $y\in Y$
\begin{equation}\label{eqW1}
\lim\limits_{t\to +\infty}|\varphi(t,u_1,y)-\varphi(t,u_2,y)|=0,
\end{equation}
where $\varphi(t,u_i,y)$ ($i=1,2$) is the solution of equation
(\ref{eqA1}) passing through $u_i$ at the initial moment $t=0$,
which is bounded on $\mathbb Z$.
\end{enumerate}

Then,
\begin{enumerate}
\item if the point $y$ is $\tau$--periodic (respectively, Bohr
almost periodic, almost automorphic, recurrent), then equation
(\ref{eqA1}) admits a unique $\tau$--periodic (respectively, Bohr
almost periodic, almost automorphic, recurrent) solution
$\varphi(t,u_{y},y)$ ($u_y\in\mathbb R^n$); \item every solution
$\varphi(t,x,y)$ is asymptotically $\tau$--periodic (respectively,
 asymptotically Bohr almost
periodic, asymptotically almost automorphic, asymptotically
recurrent)
\end{enumerate}
\end{theorem}
\begin{proof} Let $\langle \mathbb R^n,\varphi,(Y,\mathbb
Z_{+},\sigma)\rangle$ be the cocycle associated to equation
(\ref{eqA1}). Denote by $(X,\mathbb Z_{+},\pi)$ the skew-product
dynamical system, where $X:=\mathbb R^n\times Y$ and $\pi
:=(\varphi,\sigma)$ (i.e.,
$\pi(t,(u,y)):=(\varphi(t,u,y),\sigma(t,y))$ for all $x:=(u,y)\in
\mathbb R^n\times Y$ and $t\in\mathbb Z_{+}$). Consider the
non-autonomous dynamical system $\langle (X,\mathbb Z_{+},\pi),
(Y,\mathbb Z_{+},\sigma), h \rangle$ generated by the cocycle
$\varphi$ (respectively, by (\ref{eqA1}), where $h:=pr_{2}:X\to
Y$). Since $Y$ is compact, it is evident that the dynamical system
$(Y,\mathbb R_{+},\sigma)$ is compact dissipative and its Levinson
center $J_{Y}$ coincides with $Y$. By Theorem 2.23 in
\cite{Che_2004}, the skew-product dynamical system $(X,\mathbb
Z_{+},\pi)$ is compact dissipative. Denote by $J_{X}$ its Levinson
center and by $I_{y}:=pr_1(J_{X}\bigcap X_{y})$ for all $y\in Y$,
where $X_{y}:=\{x\in X:\ h(x)=y\}$. According to the definition of
the set $I_{y}\subseteq \mathbb {R}^n$ and by Theorem \ref{thCA},
$u\in I_{y}$ if and only if the solution $\varphi(t,u,y)$ is
defined on $\mathbb Z$ and bounded (i.e., the set
$\overline{\varphi(\mathbb Z,u,y)}\subseteq \mathbb R^n$ is
compact). Thus, $I_{y}=\{u\in \mathbb R^n:$ such that $ (u,y)\in
J_{X}\}$. It is easy to see that condition (\ref{eqW1}) means that
the non-autonomous dynamical system $\langle (X,\mathbb
Z_{+},\pi), (Y,\mathbb Z_{+},\sigma), h \rangle$ is weak
convergent. To finish the proof, it is sufficient to apply Lemma
3.2 and Corollary 3.10 in \cite{CC_I} for the non-autonomous
system $\langle (X,\mathbb Z_{+},\pi), (Y,\mathbb Z_{+},\sigma), h
\rangle$ generated by (\ref{eqA1}).
\end{proof}

\begin{remark} {\em Under the assumptions of Theorem \ref{thA1}, there
exists a unique almost periodic solution of  (\ref{eqA1}), but
(\ref{eqA1}) has, in general, more than one solution
defined and bounded on $\mathbb Z$. Below we will give an example
which confirms this statement.}
\end{remark}

\begin{example}\label{exWC1}{\em In our paper \cite{CC_2010} we proved that
the following system of almost periodic differential equations
\begin{equation}\label{eqWC2}
\left\{
\begin{array}{ll}
  u'=\frac{(u-\sin t)^2(2v-u+2\sin\sqrt{2}t -\sin t)+2(v-\sin \sqrt{2} t) ^5}{((u-\sin t)^2+(v-\sin\sqrt{2}t)^2)[1+((u-\sin t)^2+(v-\sin\sqrt{2}t)^2)^2]} +\cos t\\
  \\
  v'=\frac{8(v-\sin\sqrt{2}t)^2(v-u +\sin\sqrt{2}t -\sin t)}{((u-\sin t)^2+(v-\sin\sqrt{2}t)^2)[1+((u-\sin t)^2+(v-\sin\sqrt{2}t)^2)^2]}+\sqrt{2}\cos \sqrt{2} t\
\end{array}
\right.
\end{equation}
possesses the following properties:
\begin{enumerate}
\item system (\ref{eqWC2}) is dissipative; \item the
non-autonomous dynamical system $\langle (X,\mathbb
R,\tilde{\pi}), (\tilde{Y},\mathbb R,\tilde{\sigma}), h \rangle$,
generated by (\ref{eqWC2}), is weak convergent; \item
system (\ref{eqWC2}) has a unique almost periodic solution; \item system
(\ref{eqWC2}) has more than one solution defined and bounded on
$\mathbb R$.
\end{enumerate}

Now using this example and arguing as in Section \ref{sec4}
we can prove that the discretization $\langle (X,\mathbb Z,\pi),
(Y,\mathbb Z,\sigma), h \rangle$ of the dynamical system $\langle
(X,\mathbb R,\tilde{\pi}), (\tilde{Y},\mathbb R,\tilde{\sigma}), h
\rangle$ (more exactly the subsystem $\langle (X,\mathbb
R,\tilde{\pi}), (Y,\mathbb R,\tilde{\sigma}), h \rangle$ of this
system, where $Y:=H(y_0)$ and $y_0$ is some point from
$\tilde{Y}$) possesses the necessary properties. Namely,
\begin{enumerate}
\item the dynamical system $(Y,\mathbb Z,\sigma)$ is almost
periodic; \item $\langle (X,\mathbb Z,\pi), (Y,\mathbb Z,\sigma),
h \rangle$ is compact dissipative and weak convergent; \item the
Levinson center $J_{X}$ of the dynamical system $(X,\mathbb
Z,\pi)$ contains a unique almost periodic minimal set $M\subseteq
J_{X}$; \item $J_{X}\not= M$.
\end{enumerate}}
\end{example}

\subsection{Uniform compatible solutions of strict dissipative
equations}\label{sec7}

In this subsection we consider our difference equation
(\ref{eqA1}) when the driving system $(Y,\mathbb Z,\sigma)$ is
recurrent, and the function $f\in C(Y\times \mathbb R^n,\mathbb
R^n)$ is \emph{strictly contracting} with respect to its second
variable $x\in \mathbb R^n$, i.e.,
\begin{equation}\label{eqU1}
 | f(y,u_1)-f(y,u_2)|<|u_1-u_2|
 \end{equation}
for all $u_1,u_2\in \mathbb R^n$ ($u_1\not= u_2$) and $y\in Y$.

\begin{lemma}\label{lU1} Suppose that the function $f\in C(Y\times
\mathbb R^n,\mathbb R^n)$ is strictly contracting with respect to
$x\in\mathbb R^n$, then
\begin{equation}\label{eqU2}
|\varphi(t,u_1,y)-\varphi(t,u_2,y)|<|u_1-u_2|
\end{equation}
for all $u_1,u_2\in \mathbb R^n$ ($u_1\not= u_2$), $t\ge 1$ and
$y\in Y$.
\end{lemma}
\begin{proof} We will prove inequality (\ref{eqU2}) by the method of
mathematical induction. It is easy to check inequality
(\ref{eqU2}) for $n=1$. Indeed, from (\ref{eqU1}) we have
\begin{equation}\label{eqU3}
|\varphi(1,u_1,y)-\varphi(1,u_2,y)|=| f(y,u_1)-f(y,u_2)|<|u_1-u_2|
\end{equation}
for all $u_1,u_2\in \mathbb R^n$ ($u_1\not= u_2$) and $y\in Y$.
Suppose now that (\ref{eqU2}) is true for all $1\le n\le m$ and we
will show that then it is also true for $n=m+1$. Indeed, by
(\ref{eqU3}) and the inductive assumption we have
\begin{eqnarray}\label{eqU3*}
&&|\varphi(m+1,u_1,y)-\varphi(m+1,u_2,y)|\\
&&\qquad=|\varphi(1,\varphi(m,u_1,y),\sigma(1,y))- \varphi(1,\varphi(m,u_2,y),\sigma(1,y))|\notag\\
&&\qquad<|
\varphi(m,u_1,y)-\varphi(m,u_2,y)|<|u_1-u_2|\nonumber
\end{eqnarray}
for all $u_1,u_2\in \mathbb R^n$ ($u_1\not= u_2$) and $y\in Y$.
The proof is therefore complete.
\end{proof}

Recall \cite{Shch_1972,scher75,Shch_1985} that the point $x\in X$
is called \emph{comparable (respectively, uniformly comparable) by the
character of recurrence} with the point $y\in Y$ if $\mathfrak
N_{y}\subseteq \mathfrak N_{x}$ (respectively, $\mathfrak
M_{y}\subseteq \mathfrak M_{x}$), where, as defined previously, $\mathfrak
N_{x}:=\{\{t_n\}\subseteq \mathbb T:$ such that $\pi(t_n,x)\to x$
as $n\to +\infty$ $\}$ (respectively, $\mathfrak
M_{x}:=\{\{t_n\}\subseteq \mathbb T:$ such that the sequence
$\{\pi(t_n,x)\}$ converges $\}$).

Now, we can prove a result which will be helpful for our analysis.

\begin{theorem}\label{thShch} \cite{Shch_1972,Shch_1985} Let $(X,\mathbb T,\pi)$ and $(Y,\mathbb
T,\sigma)$ be two dynamical systems, $x\in X$ and $y\in Y$. Then,
the following statements hold:
\begin{enumerate}
\item If $x$ is comparable by the character of recurrence with
$y$, and $y$ is $\tau$--periodic (respectively, Levitan almost
periodic, almost recurrent, Poisson stable), then so is the point
$x$.
 \item
If $x$ is uniformly comparable by the character of recurrence with
$y$, and $y$ is $\tau$--periodic (respectively, Bohr almost
periodic, almost automorphic, recurrent), then so is the point
$x$.
\end{enumerate}
\end{theorem}

Following  B. A. Shcherbakov \cite{Shch_1972,Shch_1985}, a
solution $\varphi(t,u,y)$ of (\ref{eqA1}) is said to be
\emph{compatible (respectively, uniformly compatible) by the
character of recurrence with the right hand-side} if $\mathfrak
N_{y}\subseteq \mathfrak N_{\varphi(\cdot,u,y)}$ (respectively,
$\mathfrak M_{y}\subseteq \mathfrak M_{\varphi(\cdot,u,y)}$),
where $\mathfrak N_{\varphi(\cdot,u,y)}$ (respectively, $\mathfrak
M_{\varphi(\cdot,u,y)}$) is the family of all subsequences
$\{t_n\}\subseteq \mathbb T$ such that the sequence
$\{\varphi(t+t_n,u,y)\}$ converges to $\varphi(t,u,y)$
(respectively, converges to some function $\psi :\mathbb T\mapsto
\mathbb R^n$) uniformly with respect to $t$ on every compact from
$\mathbb T$.

\begin{theorem}\label{thPR} Let $(Y,\mathbb Z,\sigma)$ be pseudo
recurrent, $f\in C(Y\times \mathbb R^n,\mathbb R^n)$ be strict
contracting with respect to the variable $x$, and there exists at
least one solution $\varphi(t,u_0,y)$ of (\ref{eqA1}) which is
bounded on $\mathbb Z_{+}$.

Then,
\begin{enumerate}
\item system (\ref{eqA1}) is convergent, i.e., the cocycle $\varphi$
associated to (\ref{eqA1}) is convergent; \item for all $y\in Y$,
(\ref{eqA1}) admits a unique solution $\varphi(t,x_y,y)$ which is
bounded on $\mathbb Z$  and uniformly compatible, i.e., $\mathfrak
M_{y}\subseteq \mathfrak M_{\varphi(\cdot,u_y,y)}$; \item if the
point $y$ is $\tau$--periodic (respectively, Bohr almost periodic,
almost automorphic, recurrent), then
\begin{enumerate}
\item equation (\ref{eqA1}) has a unique $\tau$--periodic
(respectively, Bohr almost periodic, almost automorphic,
recurrent) solution; \item every solution $\varphi(t,u,y)$ is
asymptotically $\tau$--periodic (respectively, asymptotically Bohr
almost periodic, asymptotically almost automorphic, asymptotically
recurrent); \item $\lim\limits_{t\to
\infty}|\varphi(t,u,y)-\varphi(t,u_y,y)|=0$ for all $x\in \mathbb
R^n$ and $y\in Y$.
\end{enumerate}
\end{enumerate}
\end{theorem}
\begin{proof} Let $V: X\dot{\times}X\to \mathbb R_{+}$ be the
mapping defined by the equality $V((u_1,y),$ $(u_2,y)):=$
$|u_1-u_2|$ for all $u_1,u_2\in \mathbb R^n$ and $y\in Y$, where
$|\cdot|^2:=\langle \cdot,\cdot \rangle$ and $\langle
\xi,\eta\rangle:=\sum\limits_{i=1}^{n}\xi_{i}^{1}\xi_{i}^{2}$
($\xi^{i} :=(\xi_1^{i},\xi_{2}^{i},\dots,\xi_{n}^{i})\in \mathbb
R^n$ ($i=1,2$)). Let $(X,\mathbb Z_{+},\pi)$ be a skew-product
dynamical system associated to the cocycle $\varphi$. By Lemma
\ref{lU1} we have $V(\pi(t,(u_1,y)),$
$\pi(t,(u_2,y))<V((u_1,y),(u_2,y))$ for all $u_1,u_2\in \mathbb
R^{n}$ ($u_1\not= u_2$) and $y\in Y$. Now to prove the first
statement it is sufficient to apply Theorem 3.11 in \cite{CC_I}.

According to the first statement of the theorem, the skew-product
dynamical system $(X,\mathbb Z_{+},\pi)$ is compact dissipative
and, if $J_{X}$ is its Levinson center, then $J_{X}\bigcap X_{y}$
consists of a single point $x_y=(u_y,y)$ and $\varphi(t,u_y,y)$ is
the unique solution of (\ref{eqA1}) defined and bounded on
$\mathbb Z$. Now, we will prove that the solution
$\varphi(\cdot,u_y,y)$ is uniformly compatible by the character of
recurrence, i.e., $\mathfrak M_{y}\subseteq \mathfrak
M_{\varphi(\cdot,u_y,y)}$. It easy to see that the last statement
is equivalent to the following inclusion $\mathfrak M_{y}\subseteq
\mathfrak M_{x_y}$. Let $\{t_k\}\in \mathfrak M_{y}.$ Then, there
exists a point $q\in Y$ such that $\sigma(t_n,y)\to q$ as $n\to
\infty$. Consider the sequence $\{\pi(t_k,x_y)\}$. Since $x_y\in
J_{X}$,  the sequence $\{\pi(t_k,x_y)\}$ is relatively
compact. Let $p_1$ and $p_2$ be two points of accumulation of this
sequence. Then, there exist two subsequences
$\{t_k^{(i)}\}\subseteq \{t_k\}$ ($i=1,2$) such that
$p_{i}=\lim\limits_{k\to \infty}\pi(t_{k}^{i},x_y)$ ($i=1,2$).
Since $J_{X}$ is a compact invariant set, then $p_{i}\in J_{X}$
($i=1,2$). On the other hand,
$\pi(t_{k}^{i},x_y)=(\varphi(t_{k}^{i},u_y,y),\sigma(t_{k}^{i},y))\to
(\bar{u}_{i},q)=p_{i}$ and, consequently, $p_{i}\in X_{q}$. Thus
$p_{i}\in J_{X}\bigcap X_{q}$ ($i=1,2$) and, consequently,
$p_1=p_2$. This means that the sequence $\{\pi(t_k,x_y)\}$ is
convergent. The second statement is therefore proved.

Taking into account the first and second statements, to finish the
proof of the third statement it is sufficient to apply Theorem
\ref{thShch}.
\end{proof}

\begin{remark}  \emph{ If we replace condition (\ref{eqU1}) by a stronger
condition, then Theorem \ref{thPR} is also true without the
requirement that there exists at least one  solution which is
bounded on $\mathbb Z_{+}$. Namely, if equation (\ref{eqA1}) is uniformly contracting, which means that there exists a number
$\alpha \in (0,1)$ such that
\begin{equation}\label{eqU2*}
| f(y,u_1)-f(y,u_2)|\le \alpha |u_1-u_2|
\end{equation}
for all $u_1,u_2\in\mathbb R^n$ and $y\in Y$. The proof of this
statement will be given in the next section.}
\end{remark}

\subsection{Uniformly contracting difference equations}\label{sec8}

In this final subsection, we will prove additional results in the case in which we assume that our difference equation is uniformly contracting.

Denote by $||\gamma||:=\max\limits_{y\in Y}|\gamma(y)|$ the norm in the Banach space $C(Y,\mathbb R^n)$.

\begin{lemma}\label{lDU1} If there exists a positive number $\alpha
$ such that $|f(y,u_1)-f(y,u_2)|\le \alpha |u_1-u_2|$ (for all $y\in Y$ and $u_1,u_2\in \mathbb R^n$), then
\begin{equation}\label{eqDU1}
|\varphi(n,u_1,y)-\varphi(n,u_2,y)|\le \alpha^{n}|u_1-u_2|
\end{equation}
$\forall \ y\in Y$ and $u_1,u_2\in \mathbb R^n$.
\end{lemma}
\begin{proof} We omit the proof of this statement because it is a slight modification of the proof of Lemma
\ref{lU1}.
\end{proof}

\begin{theorem}\label{thF2}
Suppose that the following conditions are fulfilled:
\begin{enumerate}
\item the dynamical system $(Y,$ $\mathbb Z,$ $\sigma))$ is pseudo
recurrent; \item equation (\ref{eqA1}) is uniformly contracting,
i.e., there exists a number $\alpha \in (0,1)$ such that
inequality (\ref{eqU2*}) takes place.
\end{enumerate}

Then, the following statements hold:
\begin{enumerate}
\item there exists a unique mapping $\gamma\in C(Y,\mathbb R^ n)$
such that $\gamma(\sigma(t,y))=\varphi(t,$ $\gamma(y),$ $y)$ for
all $y\in Y$ and $t\in \mathbb Z_{+}$; \item the equality
\begin{equation}\label{eqUF2}
\lim\limits_{t\to
+\infty}|\varphi(t,u,y)-\varphi(t,\gamma(y),y)|=0
\end{equation}
holds for all $y\in Y$ and $u\in \mathbb R^n$.
\end{enumerate}
\end{theorem}
\begin{proof}
 Consider the cocycle $\langle \mathbb R^n , \varphi, (Y,\mathbb
Z,\sigma)\rangle$ generated by (\ref{eqA1}), where
$\varphi(t,u,y)$ is the unique solution of (\ref{eqA1}) passing
through $u\in \mathbb R^n$ at the initial moment $t=0$. For each
$k\in\mathbb Z_{+}$, we define a mapping $S^k:C(Y,\mathbb
R^n)\mapsto C(Y,\mathbb R^n)$ as
\begin{equation}\label{eqUF3}
(S^{k}\eta)(y):=\varphi(k,\eta(y),\sigma(-k,y))
\end{equation}
for all $\eta\in C(Y,\mathbb R^n)$ and $y\in Y$. Thanks to
Lemma \ref{lDU1},
\begin{eqnarray}\label{eqUD1}
 d(S^{k}\eta_{1},S^{k}\eta_{2})&=&\max\limits_{y\in
Y}|\varphi(k,\eta_{1}(y),\sigma(-k,y))-\varphi(k,\eta_{2}(y),\sigma(-k,y))|\\
&\le& \alpha^{k} \max\limits_{y\in
Y}|\eta_{1}(y)-\eta_{2}(y)|\notag\\
&=&\alpha^{k} d(\eta_1,\eta_2)\nonumber
\end{eqnarray}
for all $\eta_{1},\eta_{2}\in C(Y,\mathbb R^n)$.  It follows from
(\ref{eqUD1}) that $Lip(S^k)\le \alpha^{k}$ ($Lip(F)$
is the Lipschitz constant of $F$) and, consequently, for $k\ge 1$,
the mapping $S^{k}$ is a contraction. Since the semigroup
$\{S^k\}_{k\in\mathbb Z_{+}}$ is commutative, then it admits a
unique fixed point $\gamma$, i.e., $\gamma(\sigma
(t,y))=\varphi(t,\gamma(y),y)$ for all $y\in Y$ and $t\in \mathbb
Z_{+}$. Thus, the first statement of our theorem is proved.

The second statement follows from (\ref{eqDU1}). Indeed, we have
\begin{equation}\label{eqUF5**}
|\varphi(k,u,y)-\varphi(k,\gamma(y),y)|\le \alpha^{k}|u-\gamma(y)|
\end{equation}
for all $y\in Y$, $k\in\mathbb Z_{+}$ and $u\in\mathbb R^n$.
Passing to the limit in (\ref{eqUF5**}) we obtain the necessary
statement and the result is completely proved.
\end{proof}

From Theorem \ref{thF2} it follows the following result.

\begin{coro}\label{corU1} Under the assumptions of Theorem
\ref{thF2}, the following statements hold:

\begin{enumerate}
\item equation (\ref{eqA1}) is convergent; \item if the point
$y\in Y$ is $\tau$--periodic (respectively, almost periodic,
almost automorphic, recurrent), the equation (\ref{eqA1}) admits a
unique $\tau$--periodic (respectively, almost periodic, almost
automorphic, recurrent) solution and every solution of
(\ref{eqA1}) is asymptotically $\tau$--periodic (respectively,
asymptotically almost periodic, asymptotically almost automorphic,
asymptotically recurrent).
\end{enumerate}
\end{coro}

\begin{remark}\label{remU1}
\emph{Notice that Theorem \ref{thF2} remains true if we replace (\ref{eqU2*})
by a more general condition: there exist positive numbers
$\mathcal N,\ \nu\in (0,1)$ and a function $\omega :Y\to
(0,+\infty)$ such that
\begin{equation}\label{eqG1}
|f(y,u_1)-f(y,u_2)|\le \omega(y,|u_1-u_2|) \ (\forall y\in Y,\
u_1,u_2\in \mathbb R^n)
\end{equation}
and
\begin{equation}\label{eqG2}
\prod_{k=0}^{n-1}\omega(\sigma(k,y))\le \mathcal N \nu^{n}
\end{equation}
for all $y\in Y$ and $n\in\mathbb Z_{+}$.
 }\end{remark}

 This statement can be proved by using slight modifications of the
 proof of Theorem \ref{thF2}.

\begin{remark}\label{remU2}
{\em 1. If the dynamical system $(Y,\mathbb Z,\sigma)$ is almost
periodic (in particular, it is uniquely ergodic) and the mapping
$\omega :Y\mapsto (0,+\infty)$ is continuous, then we have (see,
for example, \cite[ChIV]{Pet_1989})
\begin{enumerate}
\item  there exists the limit
\begin{equation}\label{eqR1}
\mu :=\lim\limits_{n\to
\infty}\frac{1}{n}\sum_{k=0}^{n-1}\ln\omega(\sigma(k,y));
\end{equation}
\item this limit exists uniformly with respect to $y\in Y$;
\item the limit $\mu$ in (\ref{eqR1}) does not depend on $y\in Y$.
\end{enumerate}

2. Note that condition (\ref{eqG2}) is fulfilled if, for
example, the dynamical system $(Y,\mathbb Z,\sigma)$ is almost
periodic, $\omega\in C(Y,\mathbb R_{+})$ and the number $\mu$ in
(\ref{eqR1}) is negative.}
\end{remark}

\textbf{Acknowledgements.} We would like to thank the referees for their
interesting comments and helpful suggestions which allowed us to improve the presentation of this paper.

This paper was written while the second
author was visiting the University of Sevilla (February--September
2010) under the Programa de Movilidad de Profesores Universitarios
y Extranjeros (Ministerio de Educaci\'on, Spain) grant
SAB2009-0078. He would like to thank people of this university for
their very kind hospitality. He also gratefully acknowledges the
financial support of the
 Ministerio de Educaci\'{o}n (Spain). The first author is partially supported by grant MTM2008-00088
(Ministerio de Ciencia e Innovaci\'on, Spain) and Proyecto de
Excelencia P07-FQM02468 (Junta de Andaluc\'{\i}a, Spain).

\end{document}